\documentclass[11pt]{amsart}
\usepackage{hyperref}
\usepackage{chezbase}
\usepackage{cleveref}

\usepackage[left = 1.25in, right = 1.25in, top=1in, bottom=1in]{geometry}

\usepackage{tikz-cd}

\usepackage{amsmath}
\numberwithin{equation}{section}

\theoremstyle{plain}
\newtheorem{theorem}[equation]{Theorem}
\newtheorem{fact}[equation]{Fact}
\newtheorem*{theorem*}{Theorem}
\newtheorem*{lemma*}{Lemma}
\newtheorem*{remark*}{Remark}
\newtheorem{proposition}[equation]{Proposition}
\newtheorem{conjecture}[equation]{Conjecture}
\newtheorem{lemma}[equation]{Lemma}
\newtheorem{definition}[equation]{Definition}

\newtheorem{subtheorem}{Theorem}[equation]

\newtheorem{corollary}[subtheorem]{Corollary}

\theoremstyle{remark}
\newtheorem{remark}[equation]{Remark}

\theoremstyle{definition}

\DeclareMathOperator{\AJ}{AJ}

\title{On Fourier--Mukai Transforms of Upward Flows for Hitchin Systems}

\author{David Fang}
\address{Yale University, Department of Mathematics}
\email{david.fang@yale.edu}

\date{}

\usepackage[backend=biber,style=alphabetic]{biblatex}
\DeclareMathOperator{\el}{ell}

\usepackage{fullpage}
\usepackage{microtype}
\begin{document}

\begin{abstract}

    We consider the moduli space of semistable Higgs bundles on a smooth projective curve. Motivated by mirror symmetry, Hausel and Hitchin showed that over an open of the locus of smooth Hitchin fibers, the duality of Donagi-Pantev intertwines certain Lagrangian upward flows with hyperholomorphic vector bundles constructed from universal Higgs bundles. Using Arinkin's sheaf and some codimension estimates, we show a generalization of this result over the entire Hitchin base, for Higgs bundles of arbitrary degree. 
\end{abstract}

\maketitle

\section*{Introduction}


\subsection{Overview} Throughout, we work over the complex numbers $\CC$. Let $C$ be a nonsingular projective curve of genus $g\ge2$. The moduli space $ M_{r,d} $ of semi-stable Higgs bundles on $C$ of rank $r$ and degree $d$ is a hyperk\"ahler variety admitting a proper Lagrangian fibration $ M_{r,d} \xrightarrow{h} A $ to an affine base; this was first introduced in \cite{Hi1}, and is now known as the Hitchin system. We denote by $ M_{r,d} ^{sm}$ the locus where $h$ is smooth; the smooth Hitchin fibers are naturally identified with Jacobians of smooth curves. It was shown in \cite{DP} that for certain choice of $d$, the relative Poincar\'e line bundle $P$ over $M_{r,d}^{sm}\times_AM_{r,d}^{sm}$ induces a Fourier--Mukai equivalence of categories 
\begin{equation}\label{FMT1}
    D^b_{\Coh}( M_{r,d} ^{sm})\xrightarrow{\sim}D^b_{\Coh}( M_{r,d} ^{sm}),
\end{equation}
which may be interpreted as a classical limit of the geometric Langlands correspondence of \cite{BD}. Moreover, it is conjectured that this duality for Hitchin systems can be realized as a Fourier--Mukai transform whose kernel extends that of \cite{DP}. 

We may also interpret this duality following the Kapustin-Witten enhancements \cite{KW} to Kontsevich's homological mirror symmetry \cite{Ko}. In particular, $ M_{r,d} $ is a hyperk\"ahler variety, and homological mirror symmetry is expected to interchange ``BAA'' branes with ``BBB'' branes; roughly this may be interpreted as saying that the Fourier--Mukai transform above interchanges Lagrangian submanifolds with hyperholomorphic sheaves.

In \cite{HH}, Hausel and Hitchin introduced an example of a complex Lagrangian $W_{\delta}^+$ and a hyperholomorphic vector bundle $\Lambda_{\delta}$ which are dual over an open of $ M_{r,d_0} ^{sm}$, where $d_0=-r(r-1)(g-1)$; here by ``dual'' we mean the structure sheaf is of the Lagrangian is mapped to the vector bundle under \Cref{FMT1}. The purpose of this note is to demonstrate how this result can be extended to a larger open $\widetilde{M}_{r,d}^s\subset M_{r,d}^s$ (including the entire elliptic locus), for $d$ arbitrary.


\subsection{Generalized Poincar\'e sheaves} Here we explain more about the kernel of \Cref{FMT1} and how to generalize it. Under the spectral correspondence of \cite{BNR}, we can realize the moduli space $ M_{r,d} $ as a moduli space of one-dimensional semistable sheaves on $T^{\ast}C$, with the map $h$ sending a sheaf to its Fitting support. In this way we identify the moduli space $ M_{r,d} $ as a partial compactification of the relative Jacobian for the spectral curve $\widetilde C\rightarrow A $. When the spectral curves are integral, Arinkin constructed in \cite{Ar2} a Poincar\'e sheaf $P$ which provides a duality for the fibers. In particular, letting $ A ^{\el}$ denote the locus of integral spectral curves, this construction allows us (setting $d_0 = -r(r-1)(g-1)$) to obtain a Poincar\'e sheaf $\overline P$ on $ M_{r,d_0} ^{\el}\times_{ A } M_{r,d_0} ^{\el}$ inducing a Fourier--Mukai equivalence 
\[D^b_{\Coh}( M_{r,d_0} ^{\el})\xrightarrow{\sim}D_{\Coh}^b( M_{r,d_0} ^{\el})\]
which extends the equivalence of \cite{DP}. Note that $\overline P$ is just a maximal Cohen-Macaulay sheaf on the relative product, but not necessarily a line bundle. 

In \cite{MSY}, it was shown how to construct this equivalence over $A^{\el}$ for arbitrary degree; in this case, the Poincar\'e sheaf $P_{d,e}$ is a twisted sheaf; more precisely $P_{d,e}$ can be realized as sheaves on $\mathcal M_{r,d}^{\el}\times_{A}\mathcal M_{r,e}^{\el}$, where $\mathcal M_{r,d}^{\el}$ is a $\mu_r$-gerbe over $M_{r,d}^{\el}$; we explain more about this construction in \Cref{section:normalization}. 
Moreover, as in \cite{Li} we prove in \Cref{section:poincare} that this sheaf extends naturally to a Cohen-Macaulay sheaf $\overline P_{d,e}$ over $ \mathcal M_{r,d} ^s\times_{ A }\widetilde{\mathcal M}_{r,d} ^s$, where $\widetilde{\mathcal M}_{r,d}^s$ denotes the locus of Higgs bundles with generically regular Higgs fields, and $ M_{r,d} ^s$ is the locus of stable Higgs bundles. The main result of this paper shows that the Fourier--Mukai transform induced by $\overline P$ still sends $\mathcal O_{W_{\delta}^+}$ to $\Lambda_{\delta}$ in some possibly twisted sense, although it is not known in general whether this Fourier--Mukai transform can be extended to a derived equivalence. 

\subsection{Upward flows and mirror symmetry}
The moduli space of semistable Higgs bundles $ M_{r,d} $ admits a natural $\GG_m$ action by $\lambda\cdot (E,\phi) = (E,\lambda\phi)$. For a $\GG_m$-fixed point $[(E,\phi)]$, we define its \emph{upward flow} $W_{(E,\phi)}^+$ to be the set of points contracted to $(E,\phi)$ under the $\GG_m$-action; a Higgs bundle $(E,\phi) \in  M_{r,d} ^{s\GG_m}$ is \emph{very stable} if the corresponding upward flow is closed (this agrees with the usual definition of very stable by \cite[Proposition 2.14]{HH}). In this case, Hausel and Hitchin showed using the Bialynicki-Birula decomposition that the upward flow is a Lagrangian subvariety, naturally isomorphic to an affine space. 

From a sequence of divisors $\delta = (\delta_0, \cdots, \delta_{r-1})$ on $C$ (say with $\delta_i$ effective for $i>0$), we can construct a $\GG_m$-fixed Higgs bundle $E_{\delta} = \bigoplus_i\mathcal O_C(\delta_0 + \dots + \delta_i)\otimes K_C^{-i}$; with the nilpotent Higgs field $\phi_{\delta}$ induced by maps $b_i:\mathcal O_C(\delta_0 + \dots + \delta_{i-1}) \rightarrow\mathcal O_C(\delta_0+\dots+\delta_{i})$. If $b= b_{r-1}\circ\cdots\circ b_1$ has no repeated roots, this is very stable \cite[Theorem 1.2]{HH}, and its upward flow is a Lagrangian subvariety $W_{\delta}^+$. Hausel and Hitchin proposed a conjectural mirror of the Lagrangian $W_{\delta}^+$
\begin{conjecture}[c.f. {\cite[{}1.5]{HH}}]\label{mainconj}
    The Lagrangian subvariety $\mathcal O_{W_{\delta}^+}$ is mirror to $\Lambda_{\delta}$, the latter of which is a hyperholomorphic vector bundle whose construction is explained in \S2.1. 
\end{conjecture}


\subsection{Main results}
In \cite[Theorem 1.5]{HH}, \Cref{mainconj} is proven over some open locus $ M_{r,d} ^{\sharp}$, which lies in the locus of Higgs bundles with everywhere regular Higgs field $ M_{r,d} ^{sm}\subset  M_{r,d} $. In this paper we will formulate and prove a version of \Cref{mainconj} without any restriction on the Hitchin base. In particular, as in \cite{Li}, we note that Arinkin's construction of the Poincar\'e sheaf can be extended to a larger open locus of the relative product, say a sheaf $\overline P$ over the open subset $\widetilde{\mathcal M}_{r,d}^s\times_{ A } \mathcal M_{r,e}^s \subset  \mathcal M_{r,d} \times_{ A } \mathcal M_{r,e} $. We will show further that $W_{\delta}^+$ is contained in $\widetilde{M}_{r,d}^s$, and also that there is a canonical isomorphism $\mathcal W_{\delta}^+ := W_{\delta}^+ \times_{M_{r,d}^s}\mathcal M_{r,d}^s \rightarrow W_{\delta}^+\times B\mu_r$. A line bundle on $\mathcal W_{\delta}^+$ is just one pulled back from $B\mu_r$; when there is no ambiguity, let us denote the bundle corresponding to the character $t\mapsto t^d$ by $\mathcal O_{\mathcal W_{\delta}^+}(d)$. Our main result is the following:

\begin{theorem}\label{main1}
    Let $S_{\overline P^{\el}_{d,e}}:D^b_{\Coh}(\mathcal M_{r,d}^{\el})_{(-e)}\rightarrow D^b_{\Coh}(\mathcal M_{r,e} ^{\el})_{(d)}$ be the Fourier--Mukai transform associated to $\overline P_{d,e}^{\el}$. Then
    \[S_{\overline P^{\el}}((\mathcal O_{\mathcal W_{\delta}^+}(-e))|_{ M_{r,d} ^{\el}}) = \Lambda_{\delta}|_{ \mathcal M_{r,d} ^{\el}}.\]
\end{theorem}

In the above we only needed to use the formalism of \cite[\S4]{MSY}. When the locus of non-integral spectral curves has codimension $\ge 2$ (i.e. $r> 2$ or $g> 2$), the following comes rather easily from \Cref{main2}: 

\begin{corollary}\label{main2}
    Suppose that either $r\ne 2$ or $g>2$, and let $S_{\overline P_{d,e}}: D^b_{\QCoh}( \mathcal M_{r,d} ^s)_{(-e)}\rightarrow D^b_{\QCoh}(\widetilde{\mathcal M}_{r,e}^s)_{(d)}$ be the Fourier--Mukai transform associated to the sheaf $\overline P_{d,e}$. Then
    \[S_{\overline P_{d,e}}(\mathcal O_{\mathcal W_{\delta}^+}(-e)) = \Lambda_{\delta}|_{\widetilde{\mathcal M}_{r,d}^s}.\] 
\end{corollary}

\noindent Note that for the above functor we need to work with $\QCoh$ instead of $\Coh$, as the spaces $\widetilde{M}_{r,e}^s$ and $ M_{r,d} ^s$ are not proper over the Hitchin base. Moreover, we do not assume that $\overline P$ comes from a Fourier--Mukai equivalence.

\subsection{Acknowledgements}

We would like to thank Junliang Shen for sharing this problem with him, as well for his continued support. The author would also like to thank David Bai, Joakim F{\ae}rgeman, Soumik Ghosh, and Weite Pi for helpful discussions and comments.

\section{Background}

\subsection{Semistable Higgs bundles and the spectral correspondence}
Fix a smooth projective curve $C$ with genus $g\ge 2$, and a positive integer $r$. Recall that a \emph{Higgs bundle} on $C$ is a pair $(E,\phi)$ consisting of a vector bundle $E$ and a map $\phi:E\rightarrow E\otimes K_C$, where $K_C$ is the canonical bundle of the curve; by the \emph{rank} and \emph{degree} of a Higgs bundle we mean that of its underlying vector bundle. A Higgs bundle $(E,\phi)$ is \emph{slope semistable} (resp. stable) if for any sub-Higgs bundle $(E', \phi')$ one has
\[\frac{\deg(E')}{\rank(E')} \le \frac{\deg(E)}{\rank(E)}\qquad (\text{resp. }\frac{\deg(E')}{\rank(E')} < \frac{\deg(E)}{\rank(E)}).\] In this note we will consider the moduli space $ M_{r,d} $ of semi-stable rank $r$ degree $d$ Higgs bundles on $C$. Let 
\[ A  = \bigoplus_{i=1}^r H^0(C, K_C^{\otimes i})\] 
be the Hitchin base, and $h: M_{r,d} \rightarrow A $ be the Hitchin fibration; recall that this sends a pair $(E,\phi)$ to the characteristic polynomial of $\phi$. We denote by $ M_{r,d} ^s$ to be the open subscheme of strictly stable Higgs bundles.

Following \cite{BNR}, we may also interpret $ M_{r,d} $ as a moduli space of sheaves on a surface. The correspondence is roughly as follows: given a Higgs bundle $(E,\phi)$, we interpret the Higgs field $\phi:K_C^{-1}\rightarrow \End(E)$ as an action of $K_C^{-1}$ on the bundle $E$. This induces an action $\Sym^{\bullet}K_C^{-1}\rightarrow\End(E)$, realizing $E$ as a module over $\Sym^{\bullet}K_C^{-1}$, or equivalently a quasi-coherent sheaf on $T^{\ast}C$, which we label $R_{(E,\phi)}$. It follows that the fitting support of $R_{(E,\phi)}$ is precisely the curve defined by the characteristic polynomial of $\phi$; in this way we can view $ A $ as a moduli space of (possibly singular or nonreduced) curves in $T^{\ast}C$, which has a universal curve $\widetilde C\rightarrow A \times T^{\ast}C$. Conversely, given a sheaf supported on $\widetilde C_a$ for some $a\in  A $, its pushforward has the structure of a Higgs bundle. 

Under the spectral correspondence described above, the fiber $h^{-1}(a)$ of the Hitchin map is identified with a moduli space of sheaves supported on a spectral curve $\widetilde C_a$. The generic spectral curve is smooth and irreducible, so the corresponding fiber is identified with a component of the Picard group of the spectral curve. If $d= -r(r-1)(g-1)$, then the generic fibers are identified with the degree $0$ Jacobians of the spectral curves. For the rest of the paper we'll set $d_0 := -r(r-1)(g-1)$.  

More generally, under the spectral correspondence, line bundles on spectral curves correspond to Higgs bundles with \emph{everywhere regular} Higgs field, i.e. Higgs bundles $(E,\phi)$ such that the eigenspaces of $\phi|_c$ are all one-dimensional for all $c\in C$. We say a Higgs field is \emph{generically regular} if for all but finitely many $c\in C$, the eigenspaces of $\phi|_c$ are all one-dimensional. One checks easily that when the spectral curve is smooth, all Higgs fields are everywhere regular, and when the spectral curve is reduced, all Higgs fields are generically regular. We let $M_{r,d}^{reg}\subset M_{r,d}^s$ denote the open locus of stable Higgs bundles with everywhere regular Higgs field, and $\widetilde M_{r,d}^s\subset M_{r,d}^s$ denote the open locus of stable Higgs bundles with generically regular Higgs field.

\subsection{Universal sheaves and construction of the mirror}\label{section:normalization}
We rewrite the ideas of \cite[\S6.2]{HH} using the ideas of \cite[\S4]{MSY}. Fix for once and for all a point $c_0\in C$. We consider the stack $\mathcal M^s_{r,d}$ defined by the groupoid functor
\[\mathcal M^s_{r,d}(T) := \left\{(E, \phi, \sigma)\middle | \substack{(E,\phi)\text{ is a } T\text{-flat family of stable Higgs bundles of rank }r\text{ and degree } d\text{ on } C\times T,\\\sigma:\det(\pi_{T\ast}(E|_{c_0}))\xrightarrow{\sim}\mathcal O_T}\right\}\]

\begin{proposition}
    The functor $\mathcal M^s_{r,d}$ is represented by a Deligne-Mumford stack which is a $\mu_r$-gerbe over $M^s_{r,d}$. 
\end{proposition}

\begin{proof}
    This is the exact same as in \cite[Proposition 4.1]{MSY}; we recall it here for completeness. Let $\mathfrak M^s_{r,d}$ denote the moduli stack of stable Higgs bundles; this is a $\GG_m$-gerbe over $M^s_{r,d}$, with a universal Higgs bundle $(\mathcal E,\Phi)$ over $\widetilde C\times\mathfrak M^s_{r,d}$. The line bundle $\det(\pi_{\mathfrak M^s_{r,d}\ast}\mathcal E_{c_0})$ determines a map
    \[\mathfrak M^s_{r,d}\rightarrow B\GG_m,\]
    and the additional data $\sigma$ is obtained by taking the base change by the map $\pt\rightarrow B\GG_m$, i.e. we have a Cartesian square
    \[\begin{tikzcd}
        \mathcal M^s_{r,d} \arrow[r] \arrow[d] & \mathfrak M^s_{r,d} \arrow[d] \\
        \pt \arrow[r] & B\GG_m.
    \end{tikzcd}\]
    This shows that $\mathcal M^s_{r,d}$ is a DM stack. The fact that it is a $\mu_r$-gerbe follows from the fact that $E$ has rank $r$. 
\end{proof}
Denote the universal Higgs bundle for $\mathcal M^s_{r,d}$ by $(\EE^d, \Phi^d)$. This is a sheaf on a $\mu_r$-gerbe, or in light of \cite[Proposition 2.1.3.3]{Lieb}, a twisted sheaf on $M^s$; it follows from the definitions that the normalization for $\EE$ is the same as the one for the universal sheaf described in \cite[\S6.2]{HH}. We also denote by $\mathcal F^d\in \Coh(\widetilde C\times_A\mathcal M_{r,d}^s)$ to be the universal sheaf obtained from applying the spectral correspondence to $(\EE^d,\Phi^d)$. Note that this notation is slightly different from that of \cite[\S4]{MSY}; namely there is a shift of degree. 

We now recall the definitions of $W_{\delta}^+$ and $\Lambda_{\delta}$. Fix $\delta = (\delta_0,\delta_1,\dots,\delta_{r-1})$ is a tuple of reduced divisors with disjoint supports on $C$ so that $\delta_i$ is effective for $i\ge 1$. Let us label the supports as 
$\delta_i = c_{i1}+\cdots+c_{im_i}$, where $c_{ij}$ are points of $C$ for $i\ge 1$, and $\pm c_{0j}$ is a point of $C$. In particular, these define a Higgs bundle by 
\begin{equation}\label{eq:delta}
    E_{\delta} := \bigoplus_{i=0}^{r-1}\mathcal O_C(\delta_0+\cdots+\delta_i)\otimes K^{\otimes -i},
\end{equation}
with the Higgs field $\phi_{\delta} = \bigoplus_{i=1}^{r-1}b_i$, where 
\[b_i:\mathcal O_C(\delta_0+\cdots+\delta_{i-1})\rightarrow \mathcal O_C(\delta_0+\cdots+\delta_i)\] 
is the map induced by the divisor $\delta_i$.
By \cite[Theorem 1.2]{HH}, $(E_{\delta}, \phi_{\delta})$ is a very stable Higgs bundle, and its \emph{upward flow} 
\[W_{\delta}^+ := \{(E,\phi)\in  M_{r,d}  : \lim_{\lambda\rightarrow 0}(E,\lambda\phi) = (E_{\delta}, \phi_{\delta})\}\]
is a closed Lagrangian subvariety which is isomorphic to an affine space by \cite[\S2]{HH}. 
Note that since the Brauer group and Picard group of $W_{\delta}^+$ are trivial, in the cartesian square
\[\begin{tikzcd}
    \mathcal W_{\delta}^+ \arrow[r,hookrightarrow] \arrow[d] & \mathcal M_{r,d}^s \arrow[d]\\
    W_{\delta}^+ \arrow[r,hookrightarrow] & M_{r,d}^s
\end{tikzcd}\]
we have a canonical isomorphism of $W_{\delta}^+$-stacks
\[\mathcal W_{\delta}^+\xrightarrow{\sim} W_{\delta}^+\times B\mu_r.\]
We also remark the following:
\begin{lemma}\label{flowsupport}
    The points of $W_{\delta}^+$ are all stable and represented by Higgs bundles with generically regular Higgs fields.
\end{lemma}
\begin{proof}
    The Higgs field of $(E_{\delta}, \phi_{\delta})$ is regular away from the points of $\delta$; in particular it lies in the open set $\widetilde{M}_{r,d}^s$. On the other hand, $\widetilde{M}_{r,d}^s$ is a $\GG_m$-equivariant open, so it follows that any point $(E,\phi)\in W_{\delta}^+$ must also lie in $\widetilde{M}_{r,d}^s$. 
\end{proof}
The proposed mirror for $W_{\delta}^+$, which was first defined in \cite[\S6.2]{HH}, is constructed as follows:
\begin{equation}\label{eq:lambda}
\Lambda_{\delta} := \bigotimes_{i=1}^{r-1}\bigotimes_{j=0}^{m_i}\bigwedge^{n-i}(\EE_{c_{ij}}^d).
\end{equation}
Here when $-c_{0j}$ is a point of $C$, we set $\EE_{c_{0j}}^d := (\EE^d)^{\vee}_{-c_{0j}}$. In the case $d = d_0$, it was shown that $\Lambda_{\delta}$ is an untwisted vector bundle on $M_{r,d_0}^s$, as shown in \cite[\S6.2]{HH}.

\section{Poincar\'e sheaves}\label{section:poincare}
\noindent We review the ideas developed in \cite{Li} and \cite[\S4]{MSY}, and collect some useful lemmas regarding the Poincar\'e sheaves. We use the following lemma freely throughout this paper:
\begin{lemma}\label{lem:CMprop}
    Let $X$ be a DM stack which is Gorenstein of pure dimension, and $M$ be a maximal Cohen-Macaulay sheaf. Suppose $Z\subset X$ is a closed subscheme of codimension $\ge 2$. Then the canonical morphism $M\rightarrow j_{\ast}(M_{X\setminus Z})$ is an isomorphism, where $j:X\setminus Z\rightarrow X$ is the canonical embedding. 
\end{lemma}

\begin{proof}
    For schemes, this is a corollary of \cite[IV.2, {}5.10.5]{EGA}, as explained in \cite[{}2.2]{Ar2}. For the general case, it suffices to check \'etale locally, after which we are reduced to the case of a scheme. 
\end{proof}

\subsection{Constructions}
In \cite[Prop 3.2.3]{Li}, it is shown that, for $L$ a line bundle with $\deg L >2g$, there is a sheaf $P$ on the relative product $\widetilde{\Higgs(L)}\times_{ A }\Higgs(L)$, where $\widetilde{\Higgs(L)}$ is the moduli stack of semistable rank 2 $L$-twisted Higgs bundles with generically regular Higgs field and $\Higgs(L)$ is the moduli stack of all rank 2 $L$-twisted Higgs bundles. We argue that the same construction works in our setting. As in \cite[\S3]{Li}, we adapt the construction of Arinkin in \cite[\S4]{Ar2}.

First, by applying the spectral correspondence to the universal Higgs bundle $(\EE,\Phi)$ on $\mathcal M^s_{r,d}\times C$, we obtain a universal sheaf, which we call $\mathcal F^{d}$, on $\mathcal M^s_{r,d}\times_{ A }\widetilde C$ (note that this notation differs slightly from that of \cite{MSY}). Let $M^{reg}_{r,d}$ be the open locus of stable Higgs bundles with strictly regular Higgs field, and $\mathcal M^{reg}_{r,d}:= M^{reg}_{r,d}\times_{M^s_{r,d}}\mathcal M^{reg}_{r,d}$; these correspond to line bundles on the spectral curve. Let $p_{ij}$ be the usual projection morphisms from the stack $\widetilde C\times_{ A }\mathcal M^s_{r,d}\times_{ A }\mathcal M^s_{r,e}$; then the formula:
\begin{equation}\label{eq:regular}
    P_{d, e} := \det\R p_{23\ast}(\mathcal F^{d}\boxtimes \mathcal F^{e}) \otimes \det\R p_{23\ast}(p_{12}^{\ast}\mathcal F^{d})^{-1} \otimes \det \R p_{23\ast}(p_{13}^{\ast}\mathcal F^{e})^{-1} \otimes \det\R p_{23\ast}(p_1^{\ast}\mathcal O_{\widetilde C})
\end{equation}
defines a line bundle on $\mathcal M_{r,d}^{reg}\times_{ A }\mathcal M_{r,e}^s\cup \mathcal M_{r,d}^s\times_{ A }\mathcal M_{r,e}^{reg}$, living in the $(e, d)$-isotypic component of the Picard group. We will show that this extends to a maximal Cohen-Macaulay sheaf $\overline P_{d,e}$ on $\widetilde{\mathcal M}_{r,d}^s\times_{ A }\mathcal M^s_{r,e}$, where $\widetilde M_{r,d}^s\subset M_{r,d}^s$ is the open locus of stable Higgs bundles with generically regular Higgs field, and $\widetilde{\mathcal M}_{r,d}^s = \widetilde M_{r,d}^s\times_{M_{r,d}^s}\mathcal M_{r,d}^s$. More precisely:
\begin{proposition}
    Let $j: \mathcal M_{r,d}^{reg}\times_{ A }\mathcal M_{r,e}^s\cup \mathcal M_{r,d}^s\times_{ A }\mathcal M_{r,e}^{reg} \rightarrow \mathcal M_{r,d}^s\times_{ A }\widetilde{\mathcal M}_{r,e}^s$ be the open immersion. Then $\overline P_{d,e}:= j_{\ast}P_{d,e}$ is a maximal Cohen-Macaulay sheaf, flat over the projection to the second factor. 
\end{proposition}

\begin{proof}
    The argument is the same as in \cite[Propositions 3.2.2, 3.2.3]{Li}; in fact the statement follows immediately by pulling back the sheaf constructed there along the map $\mathcal M_{r,d}^s \rightarrow \Higgs$. We sketch the argument below. Let $\widetilde{\Hilb}_S^n$ be the isospectral Hilbert scheme constructed by Haiman in \cite{Hai} (c.f. \cite[\S3.2, 3.3]{Ar2}); this can be defined via the Cartesian product
    \[\begin{tikzcd}
        \widetilde{\Hilb_S^n} \arrow{d}{\sigma} \arrow{r}{\psi} & \Hilb_S^n \arrow[d] \\
        S^n \arrow[r] & \Sym^n(S).
    \end{tikzcd}\]
    Then consider the diagram
    \begin{equation}\label{arinkinsheaf}
        \begin{tikzcd}
            \Hilb_S^n\times \mathcal M_{r,d}^s \arrow{d}{p_1} & \widetilde{\Hilb_S^n}\times \mathcal M_{r,d}^s \arrow{l}[swap]{\psi\times \id} \arrow{r}{\sigma\times\id} & S^n\times  \mathcal M_{r,d} ^s & \widetilde C^n\times_{A}\mathcal M_{r,d}^s \arrow{l}[swap]{\iota^n\times\id}\\
            \Hilb_S^n
        \end{tikzcd}
    \end{equation}
    where $\iota:\widetilde C\hookrightarrow S\times \mathcal A$ is the embedding of the universal spectral curve into the surface $S=T^{\ast}C$, and $\widetilde C^n := \widetilde C\times_{A}\widetilde C\times\cdots\times_{ A}\widetilde C$. Set
    \begin{equation}\label{eq:poincare1}
        Q := ((\psi\times\id)_{\ast}(\sigma\times\id)^{\ast}(\iota^n\times\id)_{\ast}\mathcal (F^d)^{\boxtimes n})^{\sign}\otimes p_1^{\ast}\det(\mathcal O_{Z})^{-1},
    \end{equation}
    where $Z$ is the universal divisor for $\Hilb_S^n$. We make the following observations:
    \begin{enumerate}
        \item $Q$ is Cohen-Macaulay of codimension $n$, and flat over $\mathcal M_{r,d}^s$: this follows from the arguments of \cite[\S5]{Ar2}. 
        \item $Q$ is supported on $\Hilb_{\mathcal C/A}^n\times_A\mathcal M_{r,d}^s$: that this is true over $A^{\el}$ follows from \cite[\S5]{Ar2}. Since $Q$ is flat over $A$, this is enough to conclude the statement globally.  
    \end{enumerate}
    Now let $U\subset \Hilb_{\mathcal C/A}^n$ be the open subscheme corresponding to divisors $Z\subset C_a$ with $H^1(I_Z^{\vee}) = 0$ and $I_Z^{\vee}$ stable. Then there is an fppf cover
    \[U\xrightarrow{\varphi} \mathfrak M_{r,d}^s\]
    when $r\mid n-d$ and $n\gg0$, which induces by base change a cover
    \[\overline U \xrightarrow{\varphi}\mathcal M_{r,d}^s,\]
    where $\overline U\rightarrow U$ is a $\GG_m$-bundle. Let $\overline U^{red} = \overline U\times_AA^{red}$; then the arguments of \cite[Proposition 4.3]{Ar2} (c.f. \cite[\S4.2.2]{MRV2}) show that $Q_{\overline U^{red}}$ descends to a twisted sheaf $P_{d,e}$ on $\mathcal M^{s,red}_{r,d}\times_A\mathcal M^{s,red}_{r,e}$, which on the open $\mathcal M^{red,reg}_{r,d}\times_A\mathcal M^{s,red}_{r,e}$ agrees with \Cref{eq:regular}. Then since $\overline U\setminus \overline U^{red}$ has codimension at least $2$, we can extend the descent datum to all of $\overline U$ using \Cref{lem:CMprop}, and the result follows. 
\end{proof}

\subsection{\'Etale local structure}\label{section:etale}
Following \cite[\S4.3]{MSY}, we make some technical remarks on the local behavior of our Poincar\'e sheaves over the elliptic locus. These will be used later to reduce certain calculations on the stacks $\mathcal M_{r,d}^{\el}$ to ones on the moduli space $M_{r,d_0}^{\el}$. 

Fix an \'etale cover $U\rightarrow A^{\el}$ such that there is a section $U\hookrightarrow \widetilde C_U$; then $(M_{r,d}^{\el})_U$ can be identified as the fine moduli space for rank one torsion free sheaves on $\widetilde C_U$ normalized along $U$. For ease of notation let $\overline J^{d} = (M_{r,d}^{\el})_U, \overline{\mathcal J}^{d}= (\mathcal M_{r,d}^{\el})_U$, where $d = d+r(r-1)(g-1)$, and let $\mathcal F_{d,U}\in \Coh(\widetilde C_U\times_U\overline J^{d})$ be the universal sheaf for the relative compactified Jacobian of $\widetilde C_U$, normalized along $U$. 

As in \cite[Proposition 4.3]{MSY}, we have a map $\sigma_d:\overline{\mathcal J}^d\rightarrow\overline J^0$
defined by the sheaf
\[\mathcal G^d := \mathcal F^d\otimes p_{\widetilde C}^{\ast}\mathcal O_{\widetilde C}(-dU) \otimes p_{\overline{\mathcal J}^d}^{\ast}(\mathcal F^d\otimes p_{\widetilde C}^{\ast}\mathcal O_{\widetilde C}(-dU))|_{U\times_U\overline{\mathcal J}^d}^{\vee} \in \Coh(\widetilde C_U\times_U\overline{\mathcal J^d}).\]
This is normalized along $U$, and so defines a map $\sigma_d$ such that letting $\mathcal F \in \Coh(\widetilde C_U\times_U\overline J^0)$, we have $(\id\times\sigma_d)^{\ast}\mathcal F = \mathcal G^d$.
For notational convenience let
\[\mathcal L_d := (\mathcal F^d\otimes p_{\widetilde C}^{\ast}\mathcal O_{\widetilde C}(-dU))|_{U\times_U\overline{\mathcal J}^d}^{\vee}.\]
Then \cite[Proposition 4.3, Corollary 4.4]{MSY} and \Cref{lem:CMprop} imply:
\begin{fact}\label{localpoincare}
    The $\mathcal L^d$ are numerically trivial line bundles on $\overline{\mathcal J}^d$, and if $\overline P$ is the Poincare sheaf on $\overline J^0\times_U\overline J^0$, then we have:
    \[(\sigma_d\times_U\sigma_e)^{\ast}\overline P = \overline P_{d,e}^{\el} \otimes (L_d^{\otimes e}\boxtimes L_e^{\otimes d}).\]
    \qed
\end{fact}

\noindent The sheaf $\mathcal F^d|_{U\times\overline{\mathcal J}^{d}}$ is a twisted line bundle on $\overline{\mathcal J}^{d}$. In particular, after restriction to a point $u\in U$, the gerbe trivializes, i.e. there is an isomorphism 
\[\overline{\mathcal J}^{d}_u \xrightarrow{\sim} \overline J^{d}_u \times B\mu_r.\]
Moreover, the section $U_u$ induces isomorphisms $\overline J^d_u\rightarrow\overline J^0_u$. In particular we may choose an \'etale cover
\[\overline J^{0}_u \xrightarrow{q_d} \overline{\mathcal J}^{d}_u, \]
so that $q_d^{\ast}$ induces isomorphisms
\[D^b_{\QCoh}(\overline{\mathcal J}^{d}_u)_w \xrightarrow[\sim]{q_d^{\ast}} D^b_{\QCoh}(\overline J^{0}_u)\]
for any weight $w$. Moreover, we may choose $q_d$ such that $\sigma_d\circ q_d = \id_{\overline J^0_u}$. 
\begin{lemma}\label{fiberwise}
    Let $\overline P = \overline P_{0,0}$ be the Arinkin kernel on $\overline J^0_u\times\overline J^0_u$, and let $\Phi_F$ denote the Fourier--Mukai transform with respect to $F$. Then we have
    \[(q_d^{\ast})^{-1}\circ \Phi_{\overline P_{0,0}\otimes (L_d^e\boxtimes L_e^d)}\circ q_e^{\ast} = \Phi_{\overline P_{d,e}},\]
    where $L_d, L_e$ are numerically trivial line bundles on $\overline J^0_u$. 
\end{lemma}

\begin{proof}
Apply $(q_d\times q_e)^{\ast}$ to \Cref{localpoincare}. 
\end{proof}

\subsection{Abel--Jacobi and theorem of the square}

Consider the Abel--Jacobi map
\[\widetilde C\xrightarrow{\AJ} M_{r,1+d_0}, \qquad (x\in \widetilde C_a)\mapsto \mathfrak m_x^{\vee}.\]
After pulling back along the gerbes $\mathcal M_{r, 1-r(r-1)(g-1)}$, we obtain maps
\[\AJ: \widetilde{\mathcal C}\rightarrow \mathcal M_{r, d_0+1}\]
of $\mu_r$-gerbes, where $\widetilde{\mathcal C}$ is a $\mu_r$-gerbe over $\widetilde C$ with structure map $\sigma:\widetilde{\mathcal C}\rightarrow\widetilde C$. 
Using \cite[Proposition 4.6]{MSY} and \Cref{lem:CMprop}, we obtain:
\begin{fact}\label{AJ}
    We have
    \[(\AJ\times_{ A }\id_{\mathcal M})^{\ast}\overline P_{1,d} \simeq (\sigma\times_{ A }\id)^{\ast}\mathcal F^d \otimes p_{\widetilde{\mathcal C}}^{\ast}\mathcal N,\]
    where $\mathcal N$ is a line bundle given by a $\QQ$-divisor proportional to $D$. 
    \qed
\end{fact}

Let $M^{reg,\el}_{r,d}\subset M_{r,d}^s$ be the open locus of Higgs bundles with everywhere regular Higgs field; under the spectral correspondence these correspond to line bundles on the spectral curves. Consider the multiplication maps
\[M^{reg,\el}_{r,d_1}\times_A M_{r,d_2}^{\el} \xrightarrow{\mu} M_{r, d_1\circ d_2}^{\el},\]
where $d_1\circ d_2 := d_1+d_2 - r(r-1)(g-1)$; these induces multiplications on the gerbes
\[\mathcal M^{reg,\el}_{r,d_1}\times_A \mathcal M_{r,d_2}^{\el} \xrightarrow{\mu} \mathcal M_{r, d_1\circ d_2}^{\el}.\]
Note that we can view the latter $\mu$ as defined by the product of universal sheaves
\begin{equation}\label{eq:mu}
    \mathcal F^{d_1}\boxtimes\mathcal F^{d_2} \in \Coh(\widetilde C\times_A \mathcal M_{r,d_1}^{reg,\el}\times_A\mathcal M_{r,d_2}^{\el}),
\end{equation}
with the normalizations on $D$ induced by the ones for $\mathcal F^{d_1}$ and $\mathcal F^{d_2}$. 
Then we have: 

\begin{proposition}\label{square}
    Consider the maps
    \[\begin{tikzcd}
        \mathcal M_{r,d_1}^{reg,\el}\times_{ A }\mathcal M_{r,e}^{\el} \arrow[leftarrow]{r}{p_{13}} & \mathcal M_{r,d_1}^{reg,\el}\times_{ A }\mathcal M_{r,d_2}^{\el}\times_{ A }\mathcal M_{r,e}^{\el}  \arrow{r}{p_{23}}  \arrow{d}{\mu\times\id} & \mathcal M_{r,d_1}^{\el}\times_{ A } \mathcal M_{r,e}^{\el}\\
        & \mathcal M_{r,d_1\circ d_2}^s \times_{ A }\mathcal M_{r,e}^s. 
    \end{tikzcd}\]
    Then 
    \begin{equation}
        (\mu\times\id)^{\ast}\overline P_{d_1\circ d_2, e} \cong p_{13}^{\ast}\overline P_{d_1, e} \otimes p_{23}^{\ast}\overline P_{d_2, e}. \label{eq:square}
    \end{equation}
\end{proposition}

\begin{proof}
    First, since all sheaves are maximal Cohen-Macaulay, it suffices to show the statement after restriction to $A^{\el}$, using \Cref{lem:CMprop}. 
    
    \emph{Step 1}: We prove the statement over each fiber of $a\in A$. We achieve this by using \Cref{localpoincare} to reduce the statement to the degree $0$ case, which is proven in \cite[Proposition 6.4]{Ar2}. First, we note that for this it suffices to pass to an \'etale cover of the $A^{\el}$ as in \Cref{section:etale}. Under the notations of \Cref{section:etale}, we note that the commutativity of the following diagram follows from definitions:
    \[\begin{tikzcd}
        \mathcal J^d \times \overline{\mathcal J}^e \arrow{r}{\mu\times\id} \arrow{d}{\sigma_d\times\sigma_e} & \overline{\mathcal J}^{d\circ e} \arrow{d}{\sigma_{d\circ e}} \\
         J^0\times\overline J^0 \arrow{r}{\mu} & \overline J^0.
    \end{tikzcd}\]
    Then the following three equations are immediate from \Cref{localpoincare} and \Cref{eq:mu}:
    \begin{align*}
        (\mu\times\id)^{\ast}\overline P_{d_1\circ d_2, e} &= (\sigma_d\times\sigma_e)^{\ast}(\mu\times\id)^{\ast}\overline P_{0,0} \otimes \mu^{\ast}p_{\overline{\mathcal J}^{d\circ e}}^{\ast}L_{d_1\circ d_2} \\
        p_{13}^{\ast}P_{d_1,e} \otimes p_{23}^{\ast}\overline P_{d_2} &= (\sigma_d\times\sigma_e)^{\ast}(p_{13}^{\ast} P_{0,0}\otimes p_{23}^{\ast}\overline P_{0,0}) \otimes p_{\mathcal J^{d_1}}^{\ast}L_{d_1} \otimes p_{\overline{\mathcal J}^{d_e}}^{\ast}L_{d_2} \\
        \mu^{\ast} L_{d_1\circ d_2} &= p_{\mathcal J^{d_1}}^{\ast}L_{d_1}\circ p_{\overline{\mathcal J}^{d_2}}^{\ast}L_{d_2}
    \end{align*}
    Then the fiberwise statement reduces to the corresponding statement for $\overline P_{0,0}$, which is proven in \cite[Proposition 6.4]{Ar2}.

    \emph{Step 2}: We prove the statement after restriction along $\AJ$; this step is analogous to \cite[Proposition 6.3]{Ar2}. Namely, let $\nu$ denote the composition:
    \[\nu:\mathcal M_{r,d}^{reg,\el}\times_A\widetilde{\mathcal C} \xrightarrow{\id\times\AJ} \mathcal M_{r,d}^{reg,\el}\times_A\mathcal M_{r,d_0+1}^{\el} \xrightarrow{\mu} \mathcal M_{r,d+1}^{\el}.\]
    We show that
    \[(\nu\times\id)^{\ast} \overline P_{d+1, e} = p_{13}^{\ast}P_{d,e}\otimes p_{23}^{\ast}\mathcal F^{1}.\]
    This follows essentially from \Cref{eq:regular}. Namely, we have essentially from definitions that
    \[(\id\times\AJ)^{\ast} \mathcal F^1 = (\id\times\sigma)^{\ast}I_{\Delta} \otimes p_{\mathcal C }^\ast\mathcal N,\]
    where $\mathcal N$ is a line bundle on $\mathcal C$. A simple computation using this and \Cref{eq:mu} shows the result. 

    \emph{Step 3}: Step 1 shows us that the two sides of \Cref{eq:square} agree up to a line bundle pulled back along the projection
    \[\mathcal M_{r,d_1}^{reg,\el} \times_A \mathcal M_{r,d_2}^{\el} \times_A\mathcal M_{r,e}^{\el} \xrightarrow{p_{1}} \mathcal M_{r,d_1}^{reg,\el}.\]
    If $d_2 = d_0$, then $M_{r,d_2}\rightarrow A$ has a zero section, and \Cref{eq:regular} shows that the restriction of both sides of \Cref{eq:square} along this section are trivial, hence the result. If $d_2 = d_0+1$, then step 2 shows similarly that the two sheaves agree. 

    \emph{Step 4}: Induct on $d_2$. In particular, consider the diagram:
    \[\begin{tikzcd}
        \mathcal M_{r,d_1}^{reg,\el}\times_A \mathcal M_{r,d_2}^{reg,\el}\times_A \mathcal M_{r,d_0+1}^{\el} \times_A \mathcal M_{r,e}^{\el} \arrow{d}{\id\times\mu\times\id}\arrow{r}{\mu\times\id\times\id} & \mathcal M_{r,d_1+1}^{reg,\el}\times_A\mathcal M_{r,d_0+1}^{\el}\times_A\mathcal M_{r,e}^{\el} \arrow{d}{\mu\times\id} \\
        \mathcal M_{r,d_1}^{reg,\el}\times_A\mathcal M_{r,d_2+1}^{\el} \times_A \mathcal M_{r,e}^{\el} \arrow{r}{\mu\times\id} & \mathcal M_{r,d_1\circ d_2+1}^{\el}\times_A\mathcal M_{r,e}^{\el}.
    \end{tikzcd}\]
    Suppose the statement holds for $d_2$; then step 3 implies that
    \[(\id\times\mu\times\id)^{\ast}(\mu\times\id)\overline P_{d_1\circ d_2+1,e} \cong (\id\times\mu\times\id)^{\ast}(p_{13}^{\ast} P_{d_1,e}\otimes p_{23}^{\ast}\overline P_{d_2,e}).\]
    But its clear that for the map
    \[p_1:\mathcal M_{r,d_1}^{reg,\el}\times_A\mathcal M_{r,d_2}^{reg,\el}\times_A\mathcal M_{r,d_0+1}^{\el}\times_A\mathcal M_{r,e}^{\el} \rightarrow \mathcal M_{r,d_1}^{reg,\el},\]
    $p_1^{\ast}$ is injective on Picard groups (e.g. it's easy to check that the pushforward of the structure sheaf is the structure sheaf). Since as in step 3 we know that the two sheaves agree up to an element of $p_1^{\ast}\Pic(\mathcal M_{r,d_1}^{reg,\el})$, the result follows.
\end{proof}

\begin{remark}
    The above proof works exactly as stated in the more general setting of \cite[\S4]{MSY}, as long as $\Pic B$ is trivial. The general case works the exact same way, with a bit more book-keeping. 
\end{remark}

\begin{corollary}\label{dual}
    Let $\iota: \mathcal M_{r,d_0+k} \rightarrow \mathcal M_{r, d_0-k}$ be the map induced by $(\mathcal F_d)^{\vee}$. Then 
    \[\overline P_{d_0+k,e}|_{\mathcal M_{r,d_0+k}^{\el}\times_A\mathcal M_{r,e}^{\el}} = (\iota\times\id)^{\ast}\overline P_{d_0-k, e}^{\vee}|_{\mathcal M_{r,d_0-k}^{\el}\times_A\mathcal M_{r,e}^{\el}}\]
\end{corollary}
\begin{proof}
    It suffices to check over $\mathcal M_{r,d_0+k}^{reg,\el}\times_{A^{\el}}\mathcal M_{r,e}^{\el}$, by \Cref{lem:CMprop}. There is a commutative diagram
    \[\begin{tikzcd}
        \mathcal M_{r,d_0+k}^{reg,\el} \arrow{r}{(\id,\iota)} \arrow{dr}[swap]{h} & \mathcal M_{r,d_0+k}^{reg,\el} \times_A \mathcal M_{r,d_0-k}^{reg,\el} \arrow{r}{\mu} & \mathcal M_{r,d_0}^{reg, \el} \\
        & A \arrow{ur}[swap]{\mathcal O_{\widetilde C}} 
    \end{tikzcd}\]
    By \Cref{square}, pullback of $P_{d_0, e}$ along the top composition is $ P_{d_0+k,e}\otimes \iota^{\ast} P_{d_0-k, e}$. On the other hand, the restriction of $P_{d_0, e}$ to the zero section is just $\mathcal O_{\mathcal M_{r,e}}$, and thus the commutativity of the diagram shows that
    \[P_{d_0+k,e}\otimes \iota^{\ast}P_{d_0-k, e} = \mathcal O_{\mathcal M_{r,d_0+k}^{reg,\el}\times_A\mathcal M_{r,e}^{\el}},\]
    which implies the result. 
\end{proof}

Unlike in \cite{Li}, we do not attempt to extend $\overline P$ to a sheaf over the entire relative product. This is because $\overline P$ is enough for our purpose, by \Cref{flowsupport}

\section{Structure of upward flows}

For the reader's convenience we recall the following four statements, which are proved in \cite{HH}:

\begin{fact}[{\cite[Proposition 3.4]{HH}}]\label{hh34}
    A semistable Higgs bundle $(E,\phi)$ lies in the upward flow $W_{(\mathcal E,\Phi)}^+$ if and only if there is a filtration
    \[0 = E_0\subset E_1\subset\cdots \subset E_k = E\]
    by subbundles such that $\phi(E_i)\subset E_{i+1}\otimes K$ for all $i$, and the associated graded $(\gr(E), \gr(\phi)) = (\mathcal E, \Phi)$. Moreover, such a filtration is unique if it exists. \qed
\end{fact}

\begin{fact}[{\cite[Proposition 4.6 and Proposition 5.18]{HH}}]\label{hh518}
    Keeping the notations of previous sections, the restriction of the Hitchin fibration $h:W_{\delta}^+\rightarrow A $ is finite flat of degree $\prod_{i=1}^{r-1}{r\choose i}^{m_i}$\qed
\end{fact}

\begin{definition}
    Let $(E,\phi)$ be a Higgs bundle on $C$ and $V\subset E_c$ be a $\phi_c$-invariant subspace of the fiber $E_c$. The \emph{Hecke transform of $(E, \phi)$ at $V\subset E_c$}, denoted $\mathcal H_V(E,\phi)$, is the unique Higgs bundle $(E',\phi')$ making the following diagram commute:
    \[\begin{tikzcd}
        0 \arrow[r]&  E' \arrow[r] \arrow{d}{\phi'} & E \arrow{d}{\phi}\arrow[r] & E_c/V \arrow[r] \arrow[d] & 0 \\
        0\arrow[r] & E'\otimes K \arrow[r] & E\otimes K \arrow[r] & (E_c/V)\otimes K \arrow[r] & 0.
    \end{tikzcd}\]
\end{definition}

\begin{fact}[{\cite[Proposition 4.15]{HH}}]\label{hh415}
    Let $(E,\phi)$ be a Higgs bundle carrying a full filtration by subbundles
    \[0 = E_0 \subsetneq E_1\subsetneq\cdots\subsetneq E_r = E\]
    such that $\phi(E_i)\subset E_{i+1}\otimes K_C$. Let $c \in C$ and $V\subset E_c$ be a $\phi_c$-invariant subspace of dimension $k$. Then the Hecke transform $\mathcal H_V(E,\phi) = (E',\phi')$ has a full filtration
    \[0 = E_0'\subsetneq\cdots\subsetneq E_r' = E'\]
    such that $\phi'(E_i')\subset E_{i+1}'\otimes K_C$, and
    \[E_{i+1}'/E_i' = \begin{cases} (E_{i+1}/E_i)(-c) & i<n-k\\ (E_{i+1}/E_i) & i\ge n-k.\end{cases}\]
    Moreover the induced map $b_i':(E_i'/E_{i-1}')\rightarrow (E_{i+1}'/E_i')\otimes K$ is the same as the map $b_i:(E_i/E_{i_1})\rightarrow (E_{i+1}/E_i)$ unless $i=k$, in which case $b_i' = b_is_c$, where $s_c$ is the section of $\mathcal O(c)$. \qed
\end{fact}

Set $ A ^{\sharp}\subset A^{\el}$ to be the locus where $\div b \cup \{c_0\}$ avoids the ramification locus of the spectral curve; this is open in $ A $. Now we give a slight generalization of \cite[{}5.18]{HH}; its proof is the same as in loc. cit, but we explain it here for completeness. 
\begin{proposition}\label{sharpstructure}
    Let $a\in  A ^{\sharp}$. Then a Higgs bundle $(E, \phi)$ lies in the fiber $W_\delta^+\cap h^{-1}(a)$ if and only if, under the spectral correspondence, it corresponds to a sheaf of the form
    \[\pi_a^{\ast}(L_1)(D_1+\cdots + D_{r-1}),\]
    where $D_i\subset \pi_a^{-1}(\bigsqcup_jc_{ij})\subset C_a$ are reduced effective Cartier divisors on $C_a$, such that $|D_i\cap \pi_a^{-1}(c_{ij})| = r-i$ for all $i,j$. Moreover, $W_{\delta}^+\cap h^{-1}(a)$ is reduced, and distinct choices of tuple $(D_1,\dots, D_{r-1})$ give rise to non-isomorphic Higgs bundles. 
\end{proposition}

\begin{proof}
    Let $(E,\phi)$ be a Higgs bundle corresponding to a point in $W_{\delta}^+\cap h^{-1}(a)$. By \Cref{hh34}, there is a unique filtration $0 = E_0\subset E_1\subset\cdots\subset E_r = E$ by subbundles such that $\phi(E_i) \subset E_{i+1}\otimes K$ and $(E_{\delta}, \Phi_{\delta}) = (\gr(E), \gr(\Phi))$. Let $U\in \Coh(C_a)$ be the sheaf obtained from $(E,\phi)$ by the spectral correspondence; we have $U = \coker(\pi_a^{\ast}(E\otimes K^{-1})\xrightarrow{\pi_a^{\ast}(\Phi) - x} \pi_a^{\ast}E)$, where $x$ is the tautological section of $T^{\ast}C$. We consider the following commutative diagram:
    \begin{center}
        \begin{tikzcd}
            \pi_a^{\ast}E_i\otimes K^{-1} \arrow{r}{\pi_a^{\ast}\phi-x} \arrow[d,hookrightarrow] & \pi_a^{\ast}E_{i+1} \arrow[d,hookrightarrow] \arrow[r] & V_{i+1}'\arrow[d] \arrow[r] & 0 \\
            \pi_a^{\ast}E\otimes K^{-1} \arrow{r}{\pi_a^{\ast}\phi - x} & \pi_a^{\ast}E \arrow[r] & U \arrow[r] & 0.
        \end{tikzcd}
    \end{center}
    Let $V_i = V_i'/\text{torsion}$ be the torsion free part of $V_i'$. Since $\phi$ is generically regular, $V_{i+1}$ is a generically rank $1$ sheaf, and also torsion free by definition. Moreover, one clearly has inclusions $V_i\hookrightarrow V_{i+1}\hookrightarrow\cdots$, and $V_1 = \pi_a^{\ast}(L_1), V_n = U$. 

    We want to understand when $V_i\hookrightarrow V_{i+1}$ fails to be surjective. Let $c\in \div b_i \subset C$ be a zero of $b_i$. By assumption we know $\phi_c(E_i\otimes K|_c)\subset E_i|_c$. Since $\pi_a$ is \'etale at $c$, we know $\phi_c$ is regular semisimple, and has distinct eigenvalues. Let $\lambda$ be an eigenvalue of $\phi_c$ not contained in $E_i|_c$. Then one has
    \[(\phi_c-\lambda)(E_i\otimes K^{-1})_c \subset E_i|_c,\]
    but since this is an inclusion of vector spaces of the same rank, they are equal. Thus we have:
    \[\ker(\pi_a^{\ast}E_{i+1}|_{(c,\lambda)}\rightarrow V_{i+1}|_{(c,\lambda)}) \supset \pi_a^{\ast}(E_i\otimes K^{-1})_{(c,\lambda)} = \pi_a^{\ast}(E_i)\]
    so the map $V_i\hookrightarrow V_{i+1}$ fails to be surjective at $(c,\lambda)$ whenever $\lambda$ is an eigenvalue of $\phi_c$ not contained in $E_i|_c$. Applying this argument for all zeroes of all $b_i$, we deduce that the support of $V_{i+1}/V_i$ has at least $m_i(n-i)$ points. 

    Now we have a chain of inclusions
    \[\pi_a^{\ast}(L_1) = V_1 \hookrightarrow V_2 \hookrightarrow\cdots\hookrightarrow V_n = U\]
    such that each of the $V_i$ are torsion free sheaves generically of rank one on $C_a$. In particular, the $V_{i+1}/V_i$ are supported in dimension $0$ and thus
    \[\chi(V_n) = \chi(V_1) + \sum_{i=1}^{r-1}\ell(V_{i+1}/V_i),\]
    where $\ell$ denotes the length of the sheaf. But by Riemann-Roch and projection formula one has
    \begin{align*}
        \chi(V_n) &= \chi(U) =\chi(E) = -r(r-1)(g-1) - r(g-1) = r^2(1-g)\\
        \chi(V_1) &= \chi(\pi_a^{\ast}(L_1)) = \chi(L_1\otimes\pi_{a\ast}\mathcal O_{C_a}) = r^2(1-g)+r\deg(L_1).
    \end{align*}
    But from \Cref{eq:delta}, we know
    \[r(r-1)(1-g) = \sum_{i=0}^{r-1}(\deg L_1 + m_1 + \cdots + m_i- i(2g-2)) = r\deg L_1 + \sum_{i=1}^{r-1}m_i(r-i) + r(r-1)(1-g).\]
    Thus we have
    \[\sum\ell(V_{i+1}/V_i) = \sum_{i=1}^{r-1}m_i(r-i).\]
    Since the support of $V_{i+1}/V_i$ is at least $m_i(r-i)$ points, it follows that the support of each $V_{i+1}/V_i$ is reduced, and its support can be described fiberwise as the $n-i$ eigenvalues of $\phi_{c_{ij}}$ for each zero $c_{ij}$ of $b_i$. Let $D_i$ be the support of $V_{i+1}/V_i$; since $D_i\subset C_a^{sm}$, it is a reduced effective Cartier divisor. Since $V_1$ is a line bundle, it follows that the $V_i$ are all line bundles; then the exact sequence
    \[0 \rightarrow V_iV_{i+1}^{\vee}\rightarrow\mathcal O_{C_a}\twoheadrightarrow \mathcal O_{D_i}\rightarrow 0\]
    realizes $V_{i+1}\cong V_i(D_i)$, and thus $U = \pi_a^{\ast}(D_1+\cdots + D_{r-1})$. 

    By \Cref{hh518}, the map $W_{\delta}^+$ is finite and flat of degree $\prod_{i=1}^{r-1}{r\choose i}^{m_i}$; there are exactly this many tuples $(D_1,\dots, D_{r-1})$, where $D_i$ is a reduced divisor supported on $(r-i)$ preimages in $C_a$ of each zero of $b_i$. It thus suffices to show that each sheaf of the form $\pi_a^{\ast}(L_1)(D_1+\cdots+D_{r-1})$ describes a distinct point in $W_{\delta}^+$. 
    
    For this we proceed by induction on $\div(b) = \delta$, using the same argument as in \cite[Proof of Proposition 5.18(2)]{HH}; for this part we do not restrict the degree of $E$ or $E_{\delta}$. For $\div b= 0$ it suffices to construct an element of the upward flow; for this we can always take \cite[remark 3.8]{HH}; namely letting $a = (a_1,\dots, a_n)\in  A $, we can consider the Higgs bundle defined by: 
    \[E_L:= L\oplus L\otimes K^{-1}\oplus\cdots\oplus L\otimes K^{1-n}, \quad \phi = \begin{pmatrix} 0 & 0 & \dots & 0 & -a_n\\
    1 &0 & \dots & 0 & -a_{n-1}\\
    0 & 1 & \dots & 0 & -a_{n-2} \\
    \vdots & \vdots & \ddots & \vdots & \vdots \\
    0 & 0 & \dots & 1 & -a_1
    \end{pmatrix}.\]
    Now suppose we have constructed for each $\delta$ and each tuple of $(L_1, D_1,\dots, D_n)$ a Higgs bundle $(E,\phi)$ with a compatible filtration $E_1\subset \cdots \subset E_n$ realizing $(E,\phi)$ as an element of $W_{\delta}^+$, such that the $D_i$ are precisely the eigenvalues of the $\phi_{c_{ij}}$ not contained in $E_i$. Fix $i$ and a point $c\notin\Supp\delta$ avoiding the ramification locus of $\pi_a$, and let $\delta'_i:= \delta_i + c, \delta_j' := \delta_j$ for all $j\ne i$. Suppose $\delta'$ also defines a very stable Higgs bundle; then for each choice $I$ of $n-i$ points on $\pi_a^{-1}(c)$, we consider $(E', \phi') = \mathcal H_{V_I}(E,\phi)$ to be the Hecke transform of $(E,\phi)$, where $V_I\subset E_c$ is the subspace defined by $I$. By \Cref{hh415}, this has a (unique) filtration $E_1'\subset\cdots\subset E_r'$ realizing it as an element of the upward flow of $(E_{\delta'}, \phi_{\delta'})$, and the eigenvalues of $\phi_c'$ not contained in $E_i'$ correspond precisely to the points $I$. The uniqueness of filtration tells us that two points we obtain from different $I$ and different $(E,\phi)$ correspond to nonisomorphic Higgs bundles, since the $V_{i+1}/V_i$ are determined precisely by the subsets $I$ we picked, hence the result. 
\end{proof}

Let $D_{ij}= \pi^{-1}(c_{ij})\subset \widetilde C$ be the universal divisors (if $-c_{0j}$ is a point, we let $D_{0j} = \iota^{\ast}\pi^{-1}(-c_{0j})$). By definition, the $D_{ij}$ are \'etale over $A^{\sharp}$, so in particular the diagonal section $D_{ij}\hookrightarrow D_{ij}\times_AD_{ij}$ is a closed and open immersion. This leads to the following observation:

\begin{corollary}\label{sharpproduct}
    Under the multiplications $\mu$, the product of divisors
    \begin{equation}\label{wdivisor}
        D_{\delta}:= \prod_{i,j}(D_{ij}^{r-i}\setminus \Delta_{D_{ij}})\subset \prod_{i,j}M_{r, \pm1-r(r-1)(g-1)}
    \end{equation}
    is mapped isomorphically to $W_{\delta}^+$ over $\mathcal A^{\sharp}$, where $D_{ij}^{r-i} := \underbrace{D_{ij}\times_{ A }\cdots\times_{ A }D_{ij}}_{r-i}$, and the $\Delta$ are the diagonals of the $D_{ij}^{r-i}$.
    \qed
\end{corollary}

Note that $D_{ij}$ are also isomorphic to affine spaces, so by the same argument as in \Cref{section:normalization} we know that, letting $\mathcal D_{ij} := D_{ij}\times_{\widetilde C}\widetilde{\mathcal C}$, we have a canonical isomorphism
\[\mathcal D_{ij}\cong D_{ij}\times B\mu_r.\]
Then the corollary above, along with \Cref{square}, allows us to prove the following:

\begin{proposition}\label{sharpmirror}
    \Cref{main2} holds over $\mathcal M_{r,d_0} ^{\sharp}$, i.e. 
    \[S_{\overline P_{d,e}|_{\mathcal M_{r,d}^{\sharp}}}(\mathcal O_{\mathcal W_{\delta}^+}(-e)|_{ \mathcal M_{r,d} ^{\sharp}}) = \Lambda_{\delta}|_{ \mathcal M_{r,e} ^{\sharp}}.\]
\end{proposition}

\begin{proof}
    For this proof, we restrict all objects to $ A ^{\sharp}$. For notational convenience, for an $A$-stack $X$, we denote $X^{\sharp}:= X\times_AA^{\sharp}$, and for a sheaf $F$ on $X$, we denote by $F^{\sharp}:= F|_{X^{\sharp}}$.

    Consider the multiplication maps
    \[\begin{tikzcd}
        \underset{i,j}{\prod_A}((\mathcal D_{ij}^{\sharp})^{r-i}\setminus\Delta) \arrow[r,hookrightarrow]\arrow{d}{m} & \underset{i,j}{\prod_A} \mathcal M_{r,d_0+1}^{\sharp,reg} \arrow{d}{\mu} \\
        (\mathcal W_{\delta}^+)^{\sharp} \arrow[r,hookrightarrow] & \mathcal M_{r,d}^{\sharp,reg}.
    \end{tikzcd}.\]
    By \Cref{square}, we have
    \[(m\times\id)^{\ast}P_{d,e}^{\sharp} \cong \boxtimes_{i,j}(P_{\pm1,e}^{\sharp}\boxtimes\cdots\boxtimes P_{\pm1,e}^{\sharp})|_{(\mathcal D_{ij}^{\sharp})^{r-i}\setminus\Delta}\]
    (here we use $-1$ when $i=0$ and $-c_{0j}$ is a point, else we take $+1$). Moreover, we observe that after trivializing the gerbes as in the diagram below
    \[\begin{tikzcd}
        \underset{i,j}{\prod_A}((D_{ij}^{\sharp})^{r-i}\setminus\Delta) \arrow[r] \arrow{d}{m} & \underset{i,j}{\prod_A}((\mathcal D_{ij}^{\sharp})^{r-i}\setminus\Delta) \arrow{d}{m} \\
        (W_{\delta}^+)^{\sharp} \arrow[r] & (\mathcal W_{\delta}^+)^{\sharp},
    \end{tikzcd}\]
    $\mathcal O_{\mathcal W_{\delta}^+}(-e)^{\sharp}$ is trivial under the pullback along either composition, whence we have
    \[m^{\ast}\mathcal O_{\mathcal W_{\delta}^+}(-e) = \boxtimes \mathcal O_{\mathcal W_{\delta}^+}(-e).\]
    In particular now
    \begin{equation}\label{productformula}
        (m\times\id)^{\ast}(P_{d,e}^{\sharp}\otimes \mathcal O_{\mathcal W_{\delta}^+}(-e)^{\sharp}) = \boxtimes (P_{\pm1,e}^{\sharp}\otimes \mathcal O_{\mathcal D_{ij}}(-e)^{\sharp})|_{(D_{ij}^{\sharp})^{r-i}\setminus\Delta}.
    \end{equation}
    As usual let $p_2:\underset{i,j}{\prod_A}((D_{ij}^{\sharp})^{r-i}\setminus\Delta)\times_A\mathcal M_{r,e}^\sharp \rightarrow \mathcal M_{r,e}^\sharp$
    be the projection. Then the pushforward under $p_2$ of either side of \Cref{productformula} has an action of $\mathfrak S := \prod_{i,j}\mathfrak S_{r-i}$, where $\mathfrak S_{r-i}$ is the symmetric group which acts by permuting the terms of $((D_{i,j}^{\sharp})^{r-i}\setminus\Delta)$. Now we have
    \begin{equation}\label{product2}
        (p_{2\ast}(m\times\id)^{\ast}(P_{d,e}^{\sharp}\otimes \mathcal O_{\mathcal W_{\delta}^+}(-e)^{\sharp}))^{sign} = (p_{2\ast}\boxtimes (P_{\pm1,e}^{\sharp}\otimes \mathcal O_{\mathcal D_{ij}}(-e)^{\sharp})|_{(D_{ij}^{\sharp})^{r-i}\setminus\Delta})^{sign}.
    \end{equation}
    We treat the left hand side of \Cref{product2} first. Since $(m\times\id)$ is $\mathfrak S$-equivariant, we have
    \begin{align*}
        (p_{2\ast}(m\times\id)^{\ast}(P_{d,e}^{\sharp}\otimes \mathcal O_{\mathcal W_{\delta}^+}(-e)^{\sharp}))^{sign} &= p_{2\ast}((m\times\id)_{\ast}((m\times\id)^{\ast}P_{d,e}^{\sharp}\otimes \mathcal O_{\mathcal W_{\delta}^+}(-e)^{\sharp})^{sign})\\
        &= p_{2\ast}(P_{d,e}^{\sharp}\otimes \mathcal O_{\mathcal W_{\delta}^+}(-e)^{\sharp}\otimes ((m\times\id)_{\ast}\mathcal O_{\underset{i,j}{\prod_A}(D_{ij}^{r-i}\setminus\Delta)^{\sharp}\times_A\mathcal M_{r,e}^\sharp})^{sign}).
    \end{align*}
    But $((m\times\id)_{\ast}\mathcal O_{\underset{i,j}{\prod_A}(D_{ij}^{r-i}\setminus\Delta)\times_A\mathcal M_{r,e}^s}^{\sharp})^{sign}$ is an untwisted line bundle on $\mathcal W_{\delta}^+\times_A\mathcal M_{r,e}^\sharp$, hence it is just $\mathcal O_{\mathcal W_{\delta}^+}^{\sharp}$, and the left hand side reduces to $S_{\overline P_{d,e}^{\sharp}}(\mathcal O_{\mathcal W_{\delta}^+}(-e)^{\sharp})$. Note also that in this case we see that $S_{\overline P_{d,e}^{\sharp}}(\mathcal O_{\mathcal W_{\delta}^+}(-e)^{\sharp})$ is a twisted vector bundle of rank $\prod_{i,j}{r\choose r-i}$. 
    
    For the right hand side of \Cref{product2}, notice that, by \Cref{AJ}, we have
    \[p_{2\ast}\boxtimes(P_{\pm1,e}^{\sharp}\otimes\mathcal O_{\mathcal W_{\delta}^+}(-e)^{\sharp}) = \bigotimes_{i,j}p_{2\ast}(P_{\pm1,e}|_{D_{i,j}^{\sharp}}\otimes \mathcal O_{\mathcal W_{\delta}^+}(-e)^{\sharp}) = \bigotimes_{i,j}(\EE_{c_{ij}}^{\sharp})^{\otimes r-i},\]
    so that 
    \[p_{2\ast}\boxtimes(P_{\pm1,e}^{\sharp}\otimes\mathcal O_{\mathcal W_{\delta}^+}(-e)^{\sharp})^{sign} = \bigotimes_{i,j}\bigwedge^{r-i}\EE_{c_{ij}}^{\sharp} = \Lambda_{\delta}^{\sharp}.\]
    Note that by \Cref{dual} and the fact that the relative dualizing sheaf of $D_{ij}\rightarrow A$ is trivial, this formula works for $i=0$, regardless of sign of $c_{0j}$. Since the diagonals $\Delta$ are closed and open substacks, we have that 
    \[p_{2\ast}\boxtimes(P_{\pm1,e}^{\sharp}\otimes\mathcal O_{\mathcal W_{\delta}^+}(-e)^{\sharp})|_{(D_{ij}^{\sharp})^{r-i}\setminus\Delta} ^{sign}\rightarrow p_{2\ast}\boxtimes(P_{\pm1,e}^{\sharp}\otimes\mathcal O_{\mathcal W_{\delta}^+}(-e)^{\sharp})^{sign} \]
    is a summand. But by the computation on the left hand side these are vector bundles of the same rank, hence the the same. The result follows. 
\end{proof}

\section{Proof of Main Theorems}

\subsection{Proof over the Smooth Locus}\label{smoothproof}

\noindent The purpose of this section is to prove \Cref{main2} over the locus of smooth spectral curves. The arguments in this section follow those in \cite[prop. 5]{Ar1}. For convenience let $\widetilde{g} := r^2(g-1)+1 = \dim A $; recall that $\widetilde{g}$ is the genus of the spectral curves, and also the relative dimension of the Hitchin fibration. We first recall the form of the inverse Fourier transform, which is essentially due to \cite[Theorem 2.2]{Muk} (c.f. \cite[Theorem C]{Ar2}, \cite[Proposition 4.2]{MSY}):
\begin{equation}\label{eq:inverse}
    S_{\overline P_{d,e}^{\el}}^{-1}(\mathcal F) = p_{2\ast}(p_1^{\ast}\mathcal F\otimes (\overline P_{d,e}^{\el})^{\vee})[\widetilde{g}]. 
\end{equation}
The following lemma will also be used implicitly in the sequel:
\begin{lemma}\label{trivial}
    $M_{r,d}$ has trivial dualizing sheaf. 
\end{lemma}
\begin{proof}
    Since $M_{r,d}$ is Gorenstein, its dualizing sheaf is a line bundle; moreover, it's symplectic, so the dualizing sheaf restricted to the smooth part $M^s$ is trivial. But the complement of $M^s$ has codimension at least $2$, and the result follows. 
\end{proof}

\begin{lemma}\label{fiberlemma}
    Let $h:M\rightarrow A$ be a smooth proper fibration of smooth varieties, and $\mathcal F\in D^b_{\Coh}(M)$ such that for each closed point $a\in A$, the derived restriction $\mathcal F|_{h^{-1}(a)}$ is a dimension $0$ sheaf of constant length $r$. Then $\mathcal F$ is a sheaf which is flat over $A$, and whose support is finite over $A$. 
\end{lemma}

\begin{proof}
    Let $i_a:h^{-1}(a)\hookrightarrow M$ denote the closed immersion of the fiber. It is clear that $\mathcal F\in D^{(-\infty, 0]}(M)$. Applying $i_a^{\ast}$ to the triangle
    \[\tau_{<0} \mathcal F\rightarrow\mathcal F\rightarrow\tau_{\ge 0}\mathcal F\xrightarrow{+1},\]
    we obtain isomorphisms
    \[i_a^{\ast}\mathcal F \cong L^0i_a^{\ast}\mathcal F\cong L^0i_a^{\ast}\tau_{\ge 0}\mathcal F, \qquad L^ji_a^{\ast}\tau_{<0}\mathcal F\cong L^{j+1}i_a^{\ast}\tau_{\ge 0}\mathcal F, \forall j\ge 0.\]
    In particular $\tau_{\ge 0}\mathcal F$ is a sheaf on $M$ with finite (hence affine) support over $A$. Then by base change $h_{\ast}L^0i_{\ast}\tau_{\ge0}\mathcal F$ is in fact a rank $r$ vector bundle on $A$, hence it is $A$-flat. Thus $L^ji_a^{\ast}\tau_{\ge 0}\mathcal F = 0$ for all $j> 0$, so $L^ji_a^{\ast}\tau_{<0}\mathcal F =0$ for all fibers $a$ and all $j$, whence $\tau_{<0}\mathcal F =0$ and we win. 
\end{proof}

\begin{lemma}\label{smCM}
    The complex $S_{\overline P_{d,e}^{sm}}^{-1}(\Lambda_{\delta}|_{\mathcal M_{r,d}^{sm}})$ is a Cohen-Macaulay sheaf of codimension $\widetilde{g}$ on $\mathcal M_{r,d}^{sm}$, flat over $A$. 
\end{lemma}

\begin{proof} 
    We want to show that $S^{-1}_{\overline P^{sm}_{d,e}}(\Lambda_{\delta}|_{\mathcal M_{r,d}^{sm}})$ is a sheaf, and that its dual is a sheaf shifted in degree $\widetilde g$. But Verdier duality shows that
    \begin{align*}
        \Rshom(p_{1\ast}(p_2^{\ast}\Lambda_{\delta}\otimes \overline P_{d,e}^{\vee}|_{\mathcal M_{r,e}^{sm}})[\widetilde{g}], \mathcal O_{\mathcal M_{r,d}^{sm}}) &= p_{1\ast}\Rshom(p_2^{\ast}\Lambda_{\delta}\otimes \overline P_{d,e}^{\vee}|_{\mathcal M_{r,e}^{sm}})[\widetilde{g}], \mathcal O_{\mathcal M_{r,d}^{sm}\times_{ A }\mathcal M_{r,e}^{sm}}[\widetilde{g}])\\
        &= p_{1\ast}(p_2^{\ast}\Lambda_{\delta}^{\vee}\otimes \overline P_{d,e}|_{\mathcal M_{r,e}^{sm}}) = \iota^{\ast}S_{\overline P_{d,e}^{sm}}^{-1}(\Lambda_{\delta}^{\vee}|_{\mathcal M_{r,e}^{sm}})[-\widetilde g],
    \end{align*}
    where the last step is obtained by applying \Cref{dual}. 
    
    We first compute the supports fiberwise. Fix $a\in  A ^{sm}$, and let $J_a$ be the Jacobian of the corresponding spectral curve $C_a$. By smooth proper base change, we have
    \[S_{\overline P_{d,e}^{sm}}^{-1}(\Lambda_{\delta})|_{h^{-1}(a)} = S_{\overline P_{d,e}^{sm}|_{h^{-1}(a)\times h^{-1}(a)}}(\Lambda_{\delta}|_{h^{-1}(a)}).\]
    Pick $q_d, q_e:J_a\rightarrow h^{-1}(a)$ as in \Cref{fiberwise}, so that 
    \[S_{\overline P_{d,e}^{sm}}^{-1}(\Lambda_{\delta})|_{h^{-1}(a)} = (q_d^{\ast})^{-1}\circ S_{P_a\otimes p_1^{\ast}L_d\otimes p_2^{\ast}L_e}^{-1}\circ q_e^{\ast},\]
    where $P_a$ is the use Poincar\'e line bundle on the product of Jacobians $J_a\times J_a$, and $L_d, L_e$ are homogeneous line bundles on $J_a$.

    By \Cref{fiberlemma}, it suffices to show that $S_{P\otimes p_2^{\ast}L_e\otimes p_1^{\ast}L_d}^{-1}(q_e^{\ast}(\Lambda_{\delta}^{\pm1}))$ are supported on finite subschemes of length $\rank(\Lambda_{\delta})$. By \cite[Example 3.2]{Muk}, since the $L_e, L_d$ are homogeneous line bundles, this is the same as showing that the $q_e^{\ast}(\Lambda_{\delta}^{\pm1}|_{h^{-1}(a)})$ are homogeneous vector bundles; for this it is enough to show that the $q_e^{\ast} \EE_{c_{ij}}^e|_{h^{-1}(a)}$ are homogeneous vector bundles on $J_a$, where $\EE^e$ is the universal Higgs bundle on $C\times \mathcal M_{r,e}^s$. We'll in fact show that $(\id\times q_e)^{\ast}\EE^e|_{C_a\times h^{-1}(a)}$ is a $C$-family of homogeneous vector bundles on $J_a$. 

    Let $\pi:\widetilde C\rightarrow C\times A$ be the projection of the spectral curve; we know that $(\pi\times\id)_{\ast}\mathcal F^e = \EE^e$. But by definition of $q_e$, we have
    \[(\id\times q_e)^{\ast}(\mathcal F^e|_{C_a\times h^{-1}(a)}) = \mathcal  L_a\otimes p_{C_a}^{\ast}L_1\otimes p_{J_a}^{\ast}L_2,\] 
    where $\mathcal L_a$ is the universal line bundle on $C_a\times J_a$, $L_1$ is a line bundle on $C_a$, and $L_2$ is a torsion line bundle on $J_a$. But $L_a$ is normalized on a point of $C_a$ by definition, so $\mathcal L_a$ can be viewed as a $C_a$-family of homogeneous line bundles on $J_a$. It follows from \cite[Example 2.9 and 3.2]{Muk} that $(\pi_a\times\id)_{\ast}\mathcal L_a$ is a $C$-family of homogeneous vector bundles on $J_a$; thus the same holds for
    \[(\pi_a\times\id)_{\ast}(\mathcal L_a\otimes p_{C_a}^{\ast}L_1\otimes p_{J_a}^{\ast}L_2) = (\id\times q_e)^{\ast}\EE^e|_{h^{-1}(a)\times C_a},\]
    and the result follows. 

\end{proof}

\begin{proposition}\label{smoothmirror}
    \Cref{main2} holds over the locus of smooth spectral curves, i.e.
    \[S_{\overline P_{d,e}^{sm}}^{-1}(\Lambda_{\delta}) =\mathcal O_{\mathcal W_{\delta}^+}(-e)|_{\mathcal M_{r,d}^s}.\]
\end{proposition}

\begin{proof}
    By \Cref{smCM}, we know that $Z:= \Supp S_{\overline P_{d,e}^{sm}}^{-1}(\Lambda_{\delta}|_{ A ^{sm}})$ is finite over $ A ^{sm}$; moreover, since $S^{-1}(\Lambda_{\delta})$ is flat over $A$, each generic point of $Z$ lies over the generic point of $A^{sm}$. But by \Cref{sharpmirror}, $S^{-1}(\Lambda_{\delta})$ agrees generically $\mathcal O_{\mathcal W_{\delta}^+}\otimes \mathcal O_{\mathcal W_{\delta}^+}(-e)$; thus we conclude that $Z = \mathcal W_{\delta}^+$. Then $S_{\overline P_{d,e}^{sm}}^{-1}(\Lambda_{\delta})$ is realized as a maximal Cohen-Macaulay sheaf of generic rank $1$ on $\mathcal W_{\delta}^+$ lying in the $-e$-isotypic component of the derived category; this can only be $\mathcal O_{\mathcal W_{\delta}^+}(-e)$. 
\end{proof}

\subsection{Codimension bounds}\label{mainproof} 
To complete the proof, we more or less need to repeat the strategy employed above. 

\begin{lemma}\label{lem:CM}
    The object $S_{\overline P_{d,e}}(\mathcal O_{\mathcal W_{\delta}^+}(-e))$ is a twisted vector bundle on $\widetilde{\mathcal M}_{r,e}^s$. 
\end{lemma}

\begin{proof}
     First, notice that $\overline P_{d,e}|_{\widetilde{\mathcal M}_{r,d}^s\times_A\widetilde{\mathcal M}_{r,e}^s}$ is flat over the both factors, since \Cref{eq:regular} is symmetric. By \Cref{flowsupport} it suffices to work on this locus. But now just compute:
     \[S_{\overline P_{d,e}}(\mathcal O_{\mathcal W_{\delta}^+}(-e)) = \R p_{2\ast}(\overline P_{d,e}\otimes p_1^{\ast}\mathcal O_{\mathcal W_{\delta}^+}(-e)|_{\mathcal W_{\delta}^+\times_{ A }\widetilde{M}_{r,d}^s}).\]
     But $(\overline P_{d,e}\otimes p_1^{\ast}\mathcal O_{\mathcal W_{\delta}^+}(-e))|_{\mathcal W_{\delta}^+\times_A\widetilde{\mathcal M}^s_{r,e}}$ descends to sheaf on $W_{\delta}^+\times_A\widetilde{\mathcal M}_{r,d}^s$, which for simplicity we call $P_\delta$. 
     Moreover, the map $ W_{\delta^+}\times_{ A }\widetilde{\mathcal M}_{r,d}^s\rightarrow\widetilde{\mathcal M}_{r,d}^s$ is finite flat (in particular affine), so in fact
     \[S_{\overline P_{d,e}}(\mathcal O_{\mathcal W_{\delta}^+}(-e))|_{\widetilde{\mathcal M}_{r,e}^s}=Rp_{2\ast}P_{\delta}|_{\widetilde{\mathcal M}_{r,e}^s} = p_{2\ast}P_{\delta}|_{\widetilde{\mathcal M}_{r,e}^s},\]
     and since $P_{\delta}$ is flat over the second factor the result follows. 
\end{proof}

\noindent Now we simply need to give a codimension bound:
\begin{proposition}\label{sharpcodim}
    The complement of $ A ^\sharp\cup  A ^{sm}$ in $ A $ has codimension $\ge 2$
\end{proposition}

\begin{proof}
    Let $Z_p$ be the locus of curves ramified at $p$; then we have a cartesian diagram
    \[\begin{tikzcd}
        Z_p \arrow[r] \arrow[d] & \Delta\arrow[d] \\
         A \arrow{r}{\ev_p} & \CC^r = \{x^r + c_1x^{r-1} +\dots + c_r, c_i\in \CC\}
    \end{tikzcd}\]
    where $\ev_p:H\rightarrow \CC^r = \{x^r + c_1x^{r-1} + \cdots + c_r, c_i\in \CC\}$ is the evaluation map at $p$ and $\Delta$ is cut out by the usual discriminant. The bottom map is a linear map, and $\Delta$ is an irreducible degree $2r-2$ divisor, so the same is true for $Z_p$ as a divisor of $ A $. 

    Let $\overline S = \PP(\mathcal O_C \oplus K_C)$ be the natural compactification of $S$, and let $\overline\pi:\overline S\rightarrow C$ be the projection. The Hitchin base compactifies naturally as the linear system of divisors corresponding to the bundle $L:=\overline\pi^{\ast}(K_C)^{\otimes r}\otimes \mathcal O_{\overline S}(r)$, with a corresponding spectral curve $\overline{\widetilde C}\xrightarrow{\alpha}|L|$ compactifying the spectral curve $\widetilde C$. By \cite[III.38, 39]{Kl1}, the ramification divisor $R$ of $\widetilde C$ in $\overline S \times |L|$ is of the form
    \[[R] = \sum_{i=0}^2c_{2-i}(p_1^{\ast}\Omega_{\overline S}^1)\cdot (p_1^{\ast}c_1(L)+p_2^{\ast}\mathcal O_{|L|}(1))^{i+1}.\]
    By the argument in \cite[{}1.2]{KP}, this is the class of an irreducible divisor, and its image is the locus of singular spectral curves. We compute the degree of $p_{2\ast}[R]$ as follows: for convenience write the hyperplane class $\mathcal O_{|L|}(1) = H$. From projection formula we see that
    \begin{equation}\label{ramdivisor}
        p_{2\ast}[R] = p_{2\ast}(c_2(p_1^{\ast}\Omega_{\overline S}^1))) \cdot H + 2 p_{2\ast}(p_1^{\ast}(c_1(L)\cdot c_1(\Omega_{\overline S}^1)))\cdot H + 3p_{2\ast}(p_1^{\ast}c_1(L)^2)\cdot H.
    \end{equation}
    Write $f$ as the class of a fiber of $\overline\pi$, and $C_0$ as the class of the zero section (these are classes of $\CH^1(\overline S)$). From \cite[\S5.2]{Har}, one finds $\deg (C_0^2) = 2-2g, f^2 = 0, \deg (C_0\cdot f) = 1$. Moreover, one has $c_1(\Omega_{\overline S}^1) = -2C_0$, whence by the cotangent exact sequence
    \[0\rightarrow \overline \pi^{\ast}(K_C) \rightarrow \Omega_{\overline S}^1 \rightarrow \Omega_{\overline S/C}^1 \rightarrow 0,\]
    one has $c_1(\pi^{\ast}K_C) = (2g-2)f$, so $c_1(\Omega_{\overline S/C}^1) = -2C_0 - (2g-2)f$, and an easy computation shows $c_2(\Omega_{\overline S/C}^1) = 4-4g$. This computes the first term of \Cref{ramdivisor}. For the other terms, first note:
    \[c_1(L) = r(C_0 + (2g-2)f),\] 
    so $\deg(c_1(L)\cdot c_1(\Omega_{\overline S}^1)) = 0$; one also computes easily $\deg(c_1(L)^2) = r^2(2g-2)$. Combining this, we get $\deg(p_{2\ast}[R]) = (3r^2-2)(2g-2)$. The generic singular curve has a single nodal singularity, so the map $R \rightarrow \Im(R)$ is degree $1$ and hence the locus of singular curves on $|L|$ is an irreducible degree $(3r^2-2)(2g-2)$ divisor. On the other hand, the degree of the divisor cutting out the locus of curves ramified at $p$ is $2r-2$; thus the two irreducible divisors are distinct and their intersection is codimension $\ge2$. 
\end{proof}

\begin{proof}[Proof of \Cref{main1}, \Cref{main2}]
    First, by \cite[Proposition 4.2]{MSY}, we know that over the locus of integral spectral curves in $ A $, $S_{\overline P_{d,e}^{\el}}$ has an inverse given by \Cref{eq:inverse}. Then \Cref{smoothmirror} and \Cref{sharpmirror} show that $S_{\overline P_{d,e}^{\el}}^{-1}(\Lambda_{\delta})$ agrees with $\mathcal O_{\mathcal W_{\delta}^+}(-e)$ over $ A ^{sm}$ and $ A ^{\sharp}$, respectively. In particular $S^{-1}(\Lambda_{\delta})|_{ A ^{sm}\cup A ^{\sharp}}$ is a $(-e)$-isotypic line bundle on $\mathcal W_{\delta}^+\cap (\mathcal M_{r,d}^{sm}\cup \mathcal M_{r,d}^{\sharp})$. It follows that 
    \[S_{\overline P_{d,e}^{\el}}^{-1}(\Lambda_{\delta})|_{ A ^{sm}\cup A ^{\sharp}} = \mathcal O_{\mathcal W_{\delta}^+}(-e)|_{ A ^{sm}\cup A ^{\sharp}} \implies S_{\overline P_{d,e}^{\el}}(\mathcal O_{\mathcal W_{\delta}^+}(-e))|_{ A ^{sm}\cup A ^{\sharp}} = \Lambda_{\delta}|_{ A _{sm}\cup A ^{\sharp}}.\]
    By \Cref{lem:CM}, $S_{\overline P_{d,e}^{\el}}(\mathcal O_{\mathcal W_{\delta}^+(-e)})|_{\widetilde{\mathcal M}_{r,e}^s}$ is a vector bundle on $\widetilde{\mathcal M}_{r,e}^s$, which by above agrees with the vector bundle $\Lambda_\delta$ on $\mathcal M_{r,e}^{\sharp}\cup \mathcal M_{r,e}^{sm}$. By \Cref{sharpcodim}, this an open whose complement has codimension $\ge 2$ in $\mathcal M_{r,e}^{\el}$; \Cref{lem:CMprop} shows \Cref{main1}. When $r>2$ or $g>2$, $M_{r,e}^{\el}\subset M_{r,e}$ is an open whose complement has codimension $\ge 2$, so again \Cref{main2} follows from \Cref{lem:CMprop}. 
\end{proof}

\printbibliography

\end{document}